\documentclass[12pt]{article}
\usepackage{longtable}
\usepackage[top=1in,bottom=1in,left=1in,right=1in]{geometry}
\usepackage{hyperref}
\hypersetup{
	hypertex=true,
	colorlinks=true,
	linkcolor=blue,
	anchorcolor=blue,
	citecolor=blue
}
\usepackage{xcolor}
\usepackage{caption} % 在导言区添加

\usepackage{caption}
\usepackage{longtable}
\usepackage{array}

\usepackage{indentfirst}
\usepackage{latexsym}
\usepackage{algorithm}
\usepackage{algorithmic}
\usepackage{setspace}
\usepackage{longtable, booktabs}
\usepackage{tabu}
\usepackage{amsthm}
\usepackage{amsmath}
\usepackage{amssymb}
\usepackage{supertabular}
\usepackage{mathrsfs}

\usepackage{rotating}
\usepackage{multicol}
\usepackage{multirow}
\usepackage{makeidx} % additional command see
\usepackage{epsfig}
\usepackage{fancyhdr}       % fuer Deckblatt
\usepackage{hyperref}
\allowdisplaybreaks[4]

\newtheorem{theorem}{Theorem}[section]
\newtheorem{lemma}[theorem]{Lemma}

\newtheorem{corollary}[theorem]{Corollary}
\newtheorem{remark}[theorem]{Remark}

\def\al{\alpha}

\def\lam{\lambda}

\def\Soc{{\rm Soc}}

\def\G1{G^\mathcal{C}}

\def\D{\mathcal{D}}

\def\P{\mathcal{P}}

\DeclareMathOperator{\Aut}{Aut} 
 \DeclareMathOperator{\Out}{Out}


\begin{document}
	\title{Block-transitive $t$-($k^2,k,\lambda$) designs associated to two dimensional  projective special linear groups}
	\author{Guoqiang Xiong$^1$,Haiyan Guan$^{1,2}$\footnote{This work is supported by the National Natural Science Foundation
	of China (Grant No.12271173 ).} \\
		{\small\it  1. College of Mathematics and Physics, China Three Gorges University, }\\
		{\small\it  Yichang, Hubei, 443002, P. R. China}\\
		{\small\it 2. Three Gorges Mathematical Research Center, China Three Gorges University,}\\
		{\small\it  Yichang, Hubei, 443002, P. R. China}\\
		\date{}
	}
	
	\maketitle
	\date

\begin{abstract}
This paper investigates block-transitive automorphism groups of  $t$-$(k^2,k,\lambda)$ designs. Let $\mathcal{D}$ be a non-trivial $t$-$(k^2,k,\lambda)$ design, $G \leq \Aut(\mathcal{D})$ be block-transitive with $X\unlhd G\leq \Aut(X)$, where $X = PSL(2,q)(q\geq4)$.\! Then $q = 8$ and $\mathcal{D}$ is a $2$-$(36,6,\lambda)$ design with $\lambda \in \{2,6,9,12,18,36\}$.
	
	
\end{abstract}
	\section{Introduction}
	 A $t$-$(v,k,\lambda)$ design is defined as an incidence structure $\mathcal{D} = (\mathcal{P},\mathcal{B})$ consisting of a finite set $\mathcal{P}$ with $v$ elements, called points, and a collection $\mathcal{B}$  of $k$-element subsets of $\mathcal{P}$, called blocks, such that every $t$-element subset of $\mathcal{P}$ is contained in exactly $\lambda$ blocks. The design $\mathcal{D}$ is termed non-trivial if the parameters satisfy $t<k<v$. All $t$-$(v,k,\lambda)$ designs in this paper are assumed to be non-trivial. An automorphism of $\mathcal{D}$ is a permutation of $\mathcal{P}$ that preserves $\mathcal{B}$. The full automorphism group of $\mathcal{D}$ comprises all such automorphisms, forming a group under the operation of permutation composition, denoted by $\Aut(\mathcal{D})$  and any subgroups of  $\Aut(\mathcal{D})$ is called an automorphism group of $\mathcal{D}$. If a subgroup $G$ of $\Aut(\mathcal{D})$  acts transitively on the point set $\mathcal{P}$, then the design $\mathcal{D}$ is called point-transitive. Similarly, if G acts transitively on the block set $\mathcal{B}$, then $\mathcal{D}$ is called block-transitive. Notably, according to the result of Block(\cite{block1967orbits}), if $\mathcal{D}$ is block-transitive, it is also point-transitive. For any automorphism group $G$, $G$ is described as point-primitive if acts primitively on points $\mathcal{P}$. Conversely, it is termed point-imprimitive.

	   There has been a considerable amount of research focusing on the study of \( t \)-\((k^2, k, \lambda)\) designs.
	  In \cite{montinaro2022flag}, Montinaro and Francot proved that a flag-transitive automorphism group $G$ of a 2-$(k^2,k,\lambda)$ design  with $\lambda$ $\mid$ $k$, is either an affine group or an almost simple group. They classified the case of almost simple type in (\cite{montinaro2023classification,montinaro2022flag}), and the case of affine type in \cite{montinaro2023classifications}.
	  Recently, Guan and Zhou(\cite{guan2024reduction})  generalized their result and proved that if $G$ is a block-transitive automorphism group of a $t$-$(k^2,k,\lambda)$ design, then $G$ must act primitively on the points $\mathcal{P}$. Furthermore, $G$ is either an affine group or an almost simple group. They systematically classified block-transitive $t$-$(k^2,k,\lambda)$ designs, with $\Soc(G)$ being an alternating group or a sporadic group.  The cases where $\Soc(G)$ is a classical group or an exceptional group of Lie type are still issues to be resolved. Additionally, %Zhou and Tian(\cite{MR2833743}) proved that if $\mathcal{D}$ is a $2$-$(v,k,4)$ symmetric design admitting a flag-transitive automorphism group $G$ and $X\unlhd G\leq Aut(X)$ with socle $X\cong PSL(2,q)$, then $\mathcal{D}$ must be a $2$-$(15,8,4)$ design. Non-trivial symmetric $2$-$(v,k,\lambda)$ designs admitting a flag-transitive almost simple automorphism group whose socle is $PSL(2,q)$ have been investigated in \cite{alavi2016symmetric}.
	   Montinaro and Francot constructed four flag-transitive $2$-$(36,6,\lambda)$ designs using the group $PGL(2,8)$ in \cite[Theorem 1.1]{montinaro2022flag}. Inspired by their work, we continue to consider the block-transitive $t$-$(k^2,k,\lambda)$ designs and the two projective special linear group in this paper. Since $(\P,B^G)$ is a block-transitive $1$-design for any non-empty subset $B$ of $\P$ if $G$ is a transitive permutation group on $\P,$  we always suppose that   $t\geq2$ in this paper.  Our main result is the following.
	 \begin{theorem}\label{mt1}
	 	\textup{Let $\mathcal{D} = (\mathcal{P},\mathcal{B})$ be a non-trivial $t$-$(k^2,k,\lambda)$ design admitting a block-transitive automorphism $G$ and $X\unlhd G\leq \Aut(X)$ with $X = PSL(2,q)(q\geq4)$, let $\alpha \in \mathcal{P}$. Then $q = 8$, $X_{\alpha} \cong D_{14}$, and $\mathcal{D}$ is a   $2$-$(36,6,\lambda)$ design.}
	 \end{theorem}
	
For the case  $X$ $=$ $PSL(2,8)$ as above, we have fully constructed all  designs by using GAP(\cite{GAP4}).
	  \begin{theorem}\label{mt2}
	 	\textup{Let $\mathcal{D} = (\mathcal{P},\mathcal{B})$ be a non-trivial $2$-$(36,6,\lambda)$ design, and $G\leq \Aut(\mathcal{D})$ is block-transitive with $PSL(2,8) \unlhd G \leq P\Gamma L(2,8)$. Then  one of the following holds:}
	 	\begin{enumerate}
	 	\item[{\rm(1)}]\textup{$G = PSL(2,8)$,  $\lambda \in \{2,6,12\}$, and $\mathcal{D}$ is one of the designs listed in Table \ref{table5};}
	 	\item[{\rm(2)}]\textup{$G = P\Gamma L(2,8)$, $\lambda \in \{2,6,9,12,18,36\},$ and $\mathcal{D}$ is one of the designs listed in Table \ref{table4}.}
	 	\end{enumerate}

	 \end{theorem}
The paper is structured as follows:
 Section 2 systematically presents fundamental results concerning block-transitive
$t$-$(k^2,k,\lambda)$ designs and the two dimensional projective special linear group $PSL(2,q).$ These results serve as core theoretical tools for the proof of the main theorem and will be frequently cited in subsequent arguments.
By means of a detailed analysis of the maximal subgroups of the group
 $G$($X\unlhd G\leq \Aut(X)$ with $X$ $=$ $PSL(2,q)$), the proof of Theorem \ref{mt1} is completed in Section 3.
 Using  GAP(\cite{GAP4}), Section 4 explicitly constructs all block-transitive $2$-$(36,6,\lambda)$ designs admitting an automorphism group with socle $PSL(2,8)$.

	\section{Preliminaries}

We first collect a  well-known result about  $t$-$(v,k,\lam)$ designs.

\begin{lemma}\label{lams}\textup{ Let   $\mathcal{D}=(\mathcal{P},\mathcal{B})$ be a $t$-$(v ,k, \lam)$ design, then $\D$ is a $s$-$(v ,k, \lam_s)$ design for any $s$ with $1\leq s\leq t,$ and $$\lam_s=\lam\frac{\tbinom{v-s}{t-s}}{\tbinom{k-s}{t-s}}.$$}
\end{lemma}
  Here, $\lam_0$
 denotes the total number of blocks, which is conventionally denoted by $b.$ $\lam_1$
 represents the number of blocks incident with a given point, a quantity traditionally denoted by $r.$

	 In \cite{guan2024reduction}, Guan and Zhou  investigated  block-transitive $t$-$(k^2,k,\lambda)$ design and obtained the following fundamental result.
		\begin{lemma}\label{le5}{\rm \cite[Corollary 2.1]{guan2024reduction}}  \textup{Let $\mathcal{D} = (\mathcal{P},\mathcal{B})$ be a $t$-$(v,k,\lambda)$ design with $G\leq \Aut(\mathcal{D})$, $\alpha \in\mathcal{P}$, and let $n$ be a non-trivial subdegree of $G$. If $G$ is block-transitive, then $\frac{v-1}{k-1}$ $\mid$ $kn$. In particular, if $\mathcal{D}$ is a $t$-$(k^2,k,\lambda)$ design, then $k+1$ $\mid$ $n$.}
		\end{lemma}
		
		 \begin{lemma}\label{le4}{\rm \cite[Lemma 4.1]{guan2024reduction}} \textup{Let $\mathcal{D}$ be a $t$-$(k^2,k,\lambda)$ design admitting a block-transitive automorphism group $G$.
		  If $\alpha$ is a point of $\mathcal{P}$, then $\frac{k+1}{(k+1,| \Out(X) |)}$ divides $| X_{\alpha} |.$ }  		
		\end{lemma}
		
		The symbol $(m, n)$ is   used to denote the greatest common divisor  of two positive integers $m$ and $n$ in this paper. According to Lemma \ref{le4}, we can easily obtain the following corollary:
		\begin{corollary}\label{le10}
		\textup{Let $\mathcal{D} = (\mathcal{P},\mathcal{B})$ be a $t$-$(k^2,k,\lambda)$ design and $G\leq  \Aut({\mathcal D})$ is block-transitive,   and $\alpha \in \mathcal{P}$, then $k+1 \mid (v-1,|X_{\alpha}||\Out(X)|)$.}
		\end{corollary}
		
According to {\rm \cite[Theorem 1]{guan2024reduction}}, a block-transitive automorphism group of a $t$-$(k^2,k,\lambda)$ design must be  point-primitive, thus all maximal subgroups of $PSL(2,q)$ are needed in our discussion. Based on Dickson's classification of the subgroups of $PSL(2,q)$(\cite{MR104735}), Giudici(\cite{MR3098485}) identified all maximal subgroups of almost simple groups whose socle is $PSL(2,q)$, where $q\geq 4$. See also the results in \cite{giudici2007maximal}.
Here we list all maximal subgroups of any almost simple group with socle  $\Soc(G) = PSL(2,q)$, i.e. of all groups $G$ such that $X = PSL(2,q)\unlhd G \leq P\Gamma L(2,q)$, $q\geq4.$ 

%The classification of the maximal subgroups for $G$ can be determined by \cite{MR104735,MR1067206}, or more directly, by adopting the complete classification listed in \cite{giudici2007maximal}.
	\begin{lemma}\label{le1}{\rm \cite[Theorem 1.1]{giudici2007maximal}}
	\textup{Let $X = PSL(2,q)$	$\leq G \leq P\Gamma L(2,q)$ and let $M$ be a maximal subgroup of $G$ which does not contain $X$. Then either $M \cap X$ is maximal in $X$, or $G$ and $M$ are given in Table \ref{table:label3}.}
	%\end{enumerate}
	\end{lemma}
	
	\begin{footnotesize}
		\begin{longtabu}{cccc}
			\caption{$G$ and $M$ of Lemma \ref{le1}}
			\label{table:label3}\\
			\hline
			Line&$G$&$M$&$|G:M|$\\
			\hline
			\endfirsthead
			\hline
			Line&$G$&$M$&$|G:M|$\\
			\hline
			\endhead
			\hline
			\endfoot
			$1$&$PGL(2,7)$&$N_{G}(D_{6}) = D_{12}$&$28$\\
	     	$2$&$PGL(2,7)$&$N_{G}(D_{8}) = D_{16}$&$21$\\
	     	$3$&$PGL(2,9)$&$N_{G}(D_{10}) = D_{20}$&$36$\\
	     	$4$&$PGL(2,9)$&$N_{G}(D_{8}) = D_{16}$&$45$\\
	     	$5$&$M_{10}$&$N_{G}(D_{10}) = C_{5} \rtimes C_{4}$&$36$\\
	     	$6$&$M_{10}$&$N_{G}(D_{8}) = C_{8} \rtimes C_{2}$&$45$\\
	     	$7$&$P\Gamma L(2,9)$&$N_{G}(D_{10}) = C_{10} \rtimes C_{4}$&$36$\\
	     	$8$&$P\Gamma L(2,9)$&$N_{G}(D_{8}) = C_{8}.\Aut(C_{8})$&$45$\\
	     	$9$&$PGL(2,11)$&$N_{G}(D_{10}) = D_{20}$&$66$\\
	     	$10$&$PGL(2,q),q = p\equiv \pm11,19 \pmod{40}$&$N_{G}(A_{4}) = S_{4}$&$\frac{q(q^2-1)}{24}$\\

			\hline
		\end{longtabu}
	\end{footnotesize}
	
	 \begin{lemma}\label{le3}{\rm \cite[Theorem 2.2]{giudici2007maximal}}
		\textup{ Let $q = p^f\geq 5 $ with $p$ an odd prime. Then the maximal subgroups of $PSL(2,q)$ are:
		\begin{enumerate}
			 \item[{\rm(1)}] $C^{f}_{p} \rtimes C_{(q-1)/2}$, that is, the stabilizer of a point of the projective line,
			\item[{\rm(2)}] $D_{q-1}$, for $q \geq 13$,
			\item[{\rm(3)}] $D_{q+1}$, for $q \neq 7,9$,
			\item[{\rm(4)}] $PGL(2,q_{0})$, for $q = q_{0}^2$ ($2$ conjugacy classes),
			\item[{\rm(5)}]$PSL(2,q_{0})$, for $q = q_{0}^r$ where $r$ an odd prime,
			\item[{\rm(6)}]$A_{5}$, for $q \equiv \pm1 \pmod{10}$, where either $q = p$, or $q = p^2$ and $p \equiv \pm3 \pmod{10}$ ($2$ conjugacy classes),
			\item[{\rm(7)}]$A_{4}$, for $q = p \equiv \pm3 \pmod{8}$ and $q \not \equiv \pm 1 \pmod{10}$,
			\item[{\rm(8)}]$S_{4}$, for $q = p \equiv \pm1 \pmod{8}$ ($2$ conjugacy classes).
		\end{enumerate}  }
	\end{lemma}
	
	\begin{lemma}\label{le2}{\rm \cite[Theorem 2.1]{giudici2007maximal}}
	 \textup{Let $q = 2^f\geq 4 $. Then the maximal subgroups of $PSL(2,q)$ are:
	\begin{enumerate}
		 \item[{\rm(1)}]$C^{f}_{2} \rtimes C_{q-1}$, that is, the stabilizer of a point of the projective line,
		\item[{\rm(2)}]$D_{2(q-1)}$,
		\item[{\rm(3)}]$D_{2(q+1)}$,
		\item[{\rm(4)}]$PGL(2,q_{0})$, where $q = q_{0}^r$ for some prime $r$ and $q_{0}\neq 2$.
	\end{enumerate}}
     \end{lemma}

\section{Proof of the Theorem 1.1}
In this section, the proof of Theorem \ref{mt1} is undertaken. From here on, the assumption is made that $\mathcal{D} = (\mathcal{P},\mathcal{B})$ is a $t$-$(v,k,\lambda)$ design with $v=k^2$, $G\leq \Aut(\mathcal{D})$ is block-transitive with $X\unlhd G\leq \Aut(X)$, here $X = PSL(2,q)$, $\alpha\in \mathcal{P}$ and $b$ is  the number  of blocks. Note that, $G$ is primitive on $\mathcal{P}$(\cite[Theorem 1]{guan2024reduction}), then the point-stabilizer $G_{\alpha}$ is maximal in $G$. Notably, $G_{\alpha}$ does not contain $X$. Otherwise, $\mathcal{P} = \{\alpha\}$ for $X$ is transitive on $\mathcal{P}$. Then either $X_{\alpha}:= X \cap G_{\alpha}$ is maximal in $X$, or $(G,G_{\alpha})$ is one of pairs $(G,M)$ listed in Table \ref{table:label3}.

\subsection{Cases in Table 1}
Firstly we consider the pairs $(G,M)$ in Table \ref{table:label3}. Note that $v = k^2 = |G:G_{\alpha}| = |G:M|$ for each case. It is easily known that $|G:M|$ is not a perfect square for all cases except Lines   $3,5,7$ and $10$. For   Line 3, $G = PGL(2,9)$. Then $|G| = 720 = 2^4 \cdot 3^2\cdot 5$. Since $$b = \lambda_{2}\frac{v(v-1)}{k(k-1)} = \lambda_{2} \cdot 6 \cdot7$$ by Lemma \ref{lams}, we have $7 \mid b$, a contradiction. Similarly,   Lines 5 and 7 can be ruled out.
 For Line 10, $G = PGL(2,q)$, where $q = p \equiv \pm11, \pm19 \pmod{40}$, $|G_{\alpha}| = 24$. Since $|G_{\alpha}|$ is divided by each  subdegree, we get $k+1\mid|G_{\alpha}|$ by Lemma \ref{le5},  which implies that $k \leq 23 $. Note that $\frac{q(q^2-1)}{24} = k^2 $, then   $q \leq 23$,  thus $q = p = 11$ or $19$. Hence $k^2 = \frac{q(q^2-1)}{24} = 55$ or $285$, which is impossible.
\subsection{Odd characteristic}
In this subsection, we consider the cases listed in Lemma \ref{le3}, where $X\cap G_{\alpha}$ is a maximal subgroup of $X$, here $|X|= \frac{q(q^2-1)}{2}$. For these cases, $v = \frac{|G|}{|G_{\alpha}|} = \frac{|X|}{|X\cap G_{\alpha}|}$ by \cite[Lemma 2.2]{alavi2016symmetric}.

Case (1). $X\cap G_{\alpha} \cong C_{p}^f\rtimes C_{(q-1)/2} $.

 Here $|X\cap G_\alpha|=p^f\cdot \frac{q-1}{2}=\frac{q(q-1)}{2},$ then
 $$v = \frac{|X|}{|X\cap G_{\alpha}|} = \frac{q(q^2-1)}{2\cdot \frac{q(q-1)}{2}}  = q+1.$$

 Since $k+1 \mid v-1$, there exists an integer $i(1\leq i\leq f)$ such that $k+1 = p^i$, then $k = p^i -1$. Thus $(p^i-1)^2 = q+1$.
 It is derived that $p^{f-i} = p^i-2$, which is impossible for odd prime $p$.

Case (2). $X \cap G_{\alpha} \cong D_{q-1}$ for $q \geq 13$.

 Here $|X\cap G_{\alpha}|= q-1,$ then
 $$ v = \frac{|X|}{|X\cap G_{\alpha}|} = \frac{q(q^2-1)}{2(q-1)} = \frac{1}{2}q(q+1).$$

 Since $q = p^f$ and $p$ is an odd prime, the subdegrees are determined as follows (\cite[Table2]{MR1067206}): $1, \frac{q-1}{2}, 2(q-1)$ and $(q-1)^{\frac{(q-3)}{2}}$, where $\xi^\theta$ denotes the subdegree $\xi$ appearing with multiplicity $\theta$. Then $k+1 \mid \frac{q-1}{2}$ by Lemma \ref{le5}. Supposing that $m(k+1) = \frac{q-1}{2}$, where $m$ is a positive integer, it follows that $k = \frac{q-1}{2m}-1$. Then $$\frac{q(q+1)}{2} = \left(\frac{q-1}{2m}-1\right)^2 < \frac{(q-1)^2}{4}.$$ However, $\frac{q(q+1)}{2} > \frac{(q-1)^2}{4}$, leading to a contradiction.

 Case (3). $X \cap G_{\alpha} \cong D_{q+1}$ for $q \neq 7,9$.

  Here $|X\cap G_{\alpha}|= q+1,$ then
 $$v = \frac{|X|}{|X\cap G_{\alpha}|} = \frac{q(q^2-1)}{2(q+1)} = \frac{1}{2}q(q-1).$$

  Since $q$ is an odd prime, the subdegrees  are determined as follows (\cite[Table2]{MR1067206}): $1, \frac{q+1}{2},$ and $(q+1)^{\frac{(q-3)}{2}}$. Then $k+1\mid\frac{q+1}{2}$ by Lemma \ref{le5}. Assuming that $m(k+1) = \frac{q+1}{2}$, where $m$ is a positive integer. Hence $k = \frac{q+1}{2m}-1$, and
  $$\frac{q(q-1)}{2} = \left(\frac{q+1}{2m}-1\right)^2 < \frac{(q+1)^2}{4m^2},$$
  which implies

   $$m^2 < \frac{(q+1)^2}{2q(q-1)} < \frac{(q+1)^2}{2(q-1)^2} =\frac{1}{2}\left(1+\frac{2}{q-1}\right)^2 < 2 .$$Thus, $m = 1$, $k = \frac{q-1}{2}$, which is impossible.

Case (4). $X \cap G_{\alpha} \cong PGL(2,q_{0})$ for $q = q_{0}^2$.

Here
 $|\Out(X)| = 2f$ and $|X\cap G_{\alpha}| = |X_{\alpha}| = q_{0}(q_{0}^2-1).$
\noindent
Then $$v = \frac{|X|}{|X\cap G_{\alpha}|} = \frac{q_{0}(q_{0}^2+1)}{2},$$ and $(q_{0},v-1) = 1$. Thus
\begin{align*}
\left(v-1,2fq_{0}(q_{0}^2-1)\right)	&=\left(\frac{q_{0}(q_{0}^2+1)}{2}-1,2f(q_{0}^2-1)\right)\\
&=\left(\frac{(q_{0}-1)(q_{0}^2+q_{0}+2)}{2},2f(q_{0}^2-1)\right)\\
&=\frac{1}{2}(q_{0}-1)\left(q_{0}^2+q_{0}+2,4f(q_{0}+1)\right).\\
\end{align*}

Since  $(q_{0}^2+q_{0}+2,q_{0}+1) = 2,$ we get that $$k+1 \mid (q_{0}-1)(q_{0}^2+q_{0}+2,4f)$$  by Corollary \ref{le10}, which implies $k+1 \mid 4f(q_{0}-1)$.

Let $q_{0}=p^a$, then $f = 2a$ for $q=q_0^2$ and $q=p^f$. Thus $k+1\mid 8a(p^a-1)$, and $k^2 = \frac{p^a(p^{2a}+1)}{2}$. Hence,
$$\frac{p^a(p^{2a}+1)}{2}\leq \left(8a(p^a-1)-1\right)^2,$$
it follows that
$$p^a< 128a^2.$$
This implies that $1\leq a\leq 8.$

   Therefore, the possible pairs $(a,p)$ are: (1) $a = 1, p\leq 127$; (2) $a = 2, p\leq 19$; (3) $a = 3, p\leq 7$; (4) $a = 4, p\leq 5$; (5) $a=5, 6, 7$ or $8, p = 3.$ However,
   $v=\frac{p^a(p^{2a}+1)}{2}$ is not a   perfect square for each case.

Case (5). $X \cap G_{\alpha} \cong PSL(2,q_{0})$ for $ q =q_{0}^r$ where $r$ is an odd prime.

Here we have $|\Out(X)|=2f$ and $|X \cap G_{\alpha}| = |X_{\alpha}| = \frac{1}{2}q_{0}(q_{0}^2-1),$ then $$v = \frac{|X|}{|X\cap G_{\alpha}|} =\frac{q(q^2-1)}{q_{0}(q_{0}^2-1)} =   \frac{q_{0}^{r-1}(q_{0}^{2r}-1)}{q_{0}^2-1}.$$

According to Lemma \ref{le4}, $k+1 \mid |\Out(X)||X_{\alpha}|$, then $k\leq fq_{0}(q_{0}^2-1)-1$. Thus
$$k^2=\frac{q(q^2-1)}{q_{0}(q_{0}^2-1)} \leq\left(fq_{0}(q_{0}^2-1)-1\right)^2,$$
it follows that $q(q^2-1) < f^2q_{0}^9.$  Note that $\frac{8}{9}q^2\leq q^2-1$ and $f^2 < p^f =q$, then  $\frac{8}{9}q^2 < q_{0}^9$, i.e. $8q_{0}^{2r-9} < 9$. It implies that $r = 3$. Then $v = \frac{q_0^2(q_0^6-1)}{q_0^2-1}= q_{0}^2(q_{0}^4+q_{0}^2+1)$ and $(v-1, q_{0})=1$.
Moreover, $\left(v-1,q_{0}^2-1\right) = 2 $ for   $$v-1=q_{0}^6+q_{0}^4+q_{0}^2-1 = (q_{0}^2-1)(q_{0}^4+2q_{0}^2+3)+2.$$
Thus
 $\left(v-1,|X_{\alpha}||\Out(X)| \right)=2(v-1,f),$ hence
  $$k+1\mid 2f $$ by Corollary \ref{le10}, which implies $k \leq 2f-1$. Suppose that $q_{0} = p^a$, then $f = 3a$. Since $k^2 \leq (2f-1)^2$, we have that
$$p^{6a}<p^{2a}(p^{4a}+p^{2a}+1)\leq (6a-1)^2<36a^2.$$

 Let $F(a) = 3^{6a-2}-4a^2$, then $F'(a) = 6ln3\cdot 3^{6a-2}-8a>0$ for $a\geq 1,$ thus, $$F(a)\geq F(1)=
77>0.$$ Hence $p^{6a-2}\geq3^{6a-2}>4a^2$ for $p\geq 3.$ This implies that $p^{6a}>36a^2,$ a contradiction.

Case (6). $X\cap G_{\alpha} \cong A_{5}$, for $q \equiv \pm1 \pmod{10}$, where either $q = p$ or $q = p^2$ and $p \equiv \pm3 \pmod{10}$.

Here $|X \cap G_{\alpha}| = |X_{\alpha}| = 60$, $|\Out(X)|=2f$, and $$v = \frac{|X|}{|X\cap G_{\alpha}|} = \frac{q(q^2-1)}{120}.$$

From $k+1$ divides $|\Out(X)||X_{\alpha}| = 120f$ by Lemma \ref{le4}, then $k\leq 120f-1$. Note that $f = 1$ or $2$ for $q = p$ or $q = p^2.$  If $f = 1$, then $\frac{q(q^2-1)}{120} \leq 119^2$, we get that $q\leq 119$. Since $q = p$ is an odd prime, $q \equiv \pm1 \pmod{10}$ with $q\geq5$, the possible value for $q$ is:  $11,19,29,31,41,59,61,71,79,89,101$ or $109$. For each possible parameter, $\frac{q(q^2-1)}{120}$ is not a   perfect square. If $f = 2$, then $\frac{q(q^2-1)}{120} \leq 239^2$, implying  $q=p^2\leq 189$. Similarly,   the possible pairs $(q,p)$ are the following: (1) $q = 49, p=7$; (2) $q = 169, p = 13$. By simple calculation for $q = 49$ or $169$, $\frac{q(q^2-1)}{120}$ is not a perfect square.

Case (7). $X\cap G_{\alpha} \cong A_{4}$, where $q = p\equiv \pm3 \pmod{8}$ and $q \not\equiv \pm1 \pmod{10}$.

Here $|X \cap G_{\alpha}| = |X_{\alpha}| = 12$, $|\Out(X)|=2f=2$ for $q=p$, and $$v = \frac{|X|}{|X\cap M|} = \frac{q(q^2-1)}{24}.$$ By Lemma \ref{le4}, we get that $k+1$ divides $24$, which implies $k\leq 23$. From $\frac{q(q^2-1)}{24} \leq 23^2$, we have $q \leq 23$.
Thus, $q = 5$ or $13$ according to $q\geq5$ and $q\equiv \pm3 \pmod{8}\not\equiv \pm1 \pmod{10}.$ For $q = 5$ or $13$, $ v=\frac{q(q^2-1)}{24} = 5$ or $91$ respectively, which is impossible.

Case (8). $X\cap G_{\alpha} \cong S_{4}$, where $q = p \equiv \pm1 \pmod{8}$.

Here $|X \cap G_{\alpha}| = |X_{\alpha}| = 24$, $|\Out(X)|=2f=2$ for $q=p$, and $$v = \frac{|X|}{|X\cap G_{\alpha}|} = \frac{q(q^2-1)}{48}.$$

Since $k\leq47$ for $k+1\mid|X_\alpha||\Out(X)|$, $\frac{q(q^2-1)}{48}\leq 47^2$. Then $q\leq47$. Since $q = p \equiv \pm1 \pmod{8}$ and $q = p$ is an odd prime, the possible value of $q$ is: $7,17,23,31,41$ or $47$. For each possible case, $\frac{q(q^2-1)}{48}$ is not a  perfect square.

\subsection{Even characteristic}
In this subsection, we consider the cases listed in Lemma \ref{le2}, where $X \cap G_{\alpha}$ is a maximal subgroup of $X$, here $|X|=q(q^2-1)$.

Case (1). $X \cap G_{\alpha} \cong C_{2}^f \rtimes C_{q-1} $ for $q = 2^f \geq4$.

Here $|X\cap G_{\alpha}| = |X_{\alpha}| = q(q-1)$, then $$v = \frac{|X|}{|X\cap G_{\alpha}|} = \frac{q(q^2-1)}{ q(q-1)}= q+1.$$

According to $k+1\mid v-1$, there exists an integer $i(1\leq i \leq f,2\leq f)$ such that $k+1 =2^i$, then $k=2^i-1$. Thus $(2^i-1)^2 = q+1$. It follows that  $2^{f-i}=2^i-2$, which is equivalent to $2^{f-i} = 2$ and $2^i = 4$. Then $i=2,f = 3$. Therefore, $q=8, k = 3$, and $v=9$, namely $X = PSL(2,8)$, and $X \cap G_\alpha=C_{2}^3 \rtimes C_{7}$.

Since $\Out(PSL(2,8)) \cong Z_{3}$, $G = PSL(2,8)$ or $P\Gamma L(2,8)$. By Atlas(\cite{MR0827219}), the action of $G$ on the point set $\mathcal{P}$ is 3-transitive, where $\mathcal{P}=\{1\cdots9\}$. Thus $\mathcal{B}$ includes all 3-subsets of $\mathcal{P}.$ Therefore, $\mathcal{D}$ is trivial, a contradiction.

Case (2). $X \cap G_{\alpha} \cong D_{2(q+1)}$.

Here $|X\cap G_{\alpha}| = 2(q+1)$, then $$v = \frac{|X|}{|X\cap G_{\alpha}|} = \frac{q(q-1)}{2}.$$

Since $q$ is even, the subdegrees are as follows (\cite[Table2]{MR1067206}): $1, (q+1)^{\frac{q-2}{2}}$.  From Lemma \ref{le5}, then $k+1 \mid q+1$. Suppose that $m(k+1) = q + 1$ where $m$ is a positive integer, then $k = \frac{q+1}{m}-1$. Thus
$$\frac{q(q-1)}{2} = \left(\frac{q+1}{m}-1\right)^2 < \frac{(q+1)^2}{m^2},$$
which satisfies that
$$m^2 < \frac{2(q+1)^2}{q(q-1)} < 2\left(1 + \frac{2}{q-1}\right)^2 < 8.$$
Hence $m$ can be either $1$ or $2$. If $m = 1$, then $k = q ,$ thus $\frac{q(q-1)}{2} = q^2,$  a contradiction. If $m = 2$, then $k=\frac{q-1}{2}$ and $\frac{q(q-1)}{2} = \left(\frac{q-1}{2}\right)^2$, which is impossible for even $q$.

Case (3). $X \cap G_{\alpha} \cong D_{2(q-1)}$.

Here $|X\cap G_{\alpha}| = 2(q-1)$, then $$v = \frac{|X|}{|X\cap G_{\alpha}|} = \frac{q(q+1)}{2}.$$

The subdegrees  are identified as (\cite[Table2]{MR1067206}): $1,2(q-1),(q-1)^{\frac{q-2}{2}}$. Hence $k+1 \mid q-1$ by Lemma \ref{le5}. Let $m(k+1) = q-1$ where $m$ is a positive integer, then $k = \frac{q-1}{m}-1$. Thus,
$$\frac{q(q+1)}{2} = \left(\frac{q-1}{m}-1\right)^2 < \frac{(q-1)^2}{m^2}.$$
This implies
$$m^2 < \frac{2(q-1)^2}{q(q+1)} < 2\left(1-\frac{1}{q}\right)^2 < 2.$$
Therefore, $m = 1, k = q-2$. Since $\frac{q(q+1)}{2} = (q-2)^2$ and $q = 2^f$, we get $q = 8, k = 6, v = 36$, namely $X = PSL(2,8)$, and $X\cap G_{\alpha} \cong D_{14}$. Then $G = PSL(2,8)$ or $P\Gamma L(2,8),$
 and the order of $G$ is $ 504$ or $1512$ respectively.
  If $t\geq3,$ then
 $$b=\lambda_{3}\cdot\frac{36\cdot35\cdot34}{6\cdot5\cdot 4} = \lambda_3\cdot3\cdot7\cdot17,$$
 by Lemma \ref{lams},
 which is impossible for $b\mid|G|$. Thus  $t=2.$

Case (4). $X \cap G_{\alpha} \cong PGL(2,q_{0})$, where $q = q_{0}^r$ for some prime $r$ and $q_{0} \neq 2$.

Here we have $|\Out(X)| = f$, and  $|X \cap G_{\alpha}| = |X_{\alpha}| = q_{0}(q_{0}^2-1),$ then $$v = \frac{|X|}{|X\cap G_{\alpha}|}= \frac{q_{0}^{r-1}(q_{0}^{2r}-1)}{q_{0}^2-1} .$$

 Thus $k+1 \mid fq_{0}(q_{0}^2-1)$ by Lemma \ref{le4}, which implies
 $$k^2=\frac{q(q^2-1)}{q_{0}(q_{0}^2-1)} \leq\left(fq_{0}(q_{0}^2-1)-1\right)^2.  $$
 It follows that  $$q(q^2-1) < f^2q_{0}^9.$$
 Since $\frac{8}{9}q^2\leq q^2-1 $ and $f^2 < p^f =q$, we have $\frac{8}{9}q^2 < q_{0}^9$, i.e. $8q_{0}^{2r-9} < 9$. It implies that $r = 2,3$.

 If $r = 3$, then $v = q_{0}^2(q_{0}^4+q_{0}^2+1)$ and $(v-1, q_{0})=1$.
Besides, $$v-1 = q_{0}^6+q_{0}^4+q_{0}^2-1 = (q_{0}^2-1)(q_{0}^4+2q_{0}^2+3)+2,$$ then $$(v-1,q_{0}^2-1)=\left(2,q_{0}^2-1\right) = 1$$ for
     $q_{0}^2-1$ is odd.  It follows that $k+1\mid f$ by Corollary \ref{le10}, which implies $k \leq f-1$. Suppose that $q_{0} = 2^a(1<a<f)$, then $f = 3a$. Note that $k^2 \leq (f-1)^2$, thus
 $$2^{6a}<2^{2a}(2^{4a}+2^{2a}+1)\leq (3a-1)^2<9a^2.$$

 Let $F(a) = 2^{6a}-9a^2$, then $F'(a) = 6ln2\cdot 2^{6a}-18a>0$ for $a > 1$, thus, $$F(a)>F(1)=
 55>0.$$ This implies that $2^{6a}>9a^2$, a contradiction.

 If $r = 2$, then  $$v = \frac{q_{0}(q_{0}^4-1)}{q_{0}^2-1} = q_{0}(q_{0}^2+1),$$ and $(q_{0},v-1)=1$. Since $$v-1 = q_{0}^3+q_{0}-1 = (q_{0}^2-1)q_{0}+2q_{0}-1,$$ and $$4(q_0^2-1)=(2q_0+1)(2q_0-1)-3,$$
  it follows that $\left(v-1,q_{0}^2-1\right) =(2q_{0} -1,3).$
   By Corollary \ref{le10},  we get that $$\left(v-1,|X_\alpha||\Out(X)|\right)\mid 3f.$$ This implies $k+1\mid 3f$. Thus $k \leq 3f-1$. Then
 $$2^{3a}<2^{a}(2^{2a}+1)\leq (6a-1)^2<36a^2$$
 for $f=2a$.
Apply the same method as used above, we get that $a=2$. However, $$v=q_0(q_0^2+1)=2^a(2^{2a}+1)=68$$ is not a perfect square.
\medskip

\noindent{\bf Proof of Theorem \ref{mt1}} Based on the analysis of sections 3.1-3.3, the Theorem \ref{mt1} holds.$\hfill\square$

 \section{Designs with $X = PSL(2,8)$}

 Montinaro and  Francot proved that there are four flag-transitive $2$-$(36,6,\lambda)$ designs admitting $PGL(2,8)$ as an automorphism group by considering the
action of this group on the desarguesian plane of order 8(\cite[Example 3.9]{montinaro2022flag}).
 In this section, we will directly construct  all block-transitive  $2$-$(36,6,\lambda)$ designs admitting an automorphism group with socle being $PSL(2,8)$ with the aid of the software GAP(\cite{GAP4}).

 \begin{lemma}\label{le7}
 	 \textup{Let $\mathcal{D}=(\mathcal{P},\mathcal{B})$ be a non-trivial $2$-$(36,6,\lambda)$ design admitting $PSL(2,8)$ as a block-transitive automorphism group, then $\lambda \in \{2,6,12\}.$   All the 46 possible designs are those listed in Table \ref{table5}.}
 \end{lemma}
\begin{proof} Let $\alpha\in \P, G=PSL(2,8)$, then
 $G_{\alpha} \cong D_{14}$. Let $\mathcal{P} = \{1,2,\cdots,36\}$ and $\mathcal{P}(G)$ be the image of the permutation representations of $G$ on $\mathcal{P}$. Then $\mathcal{P}(G) = \left\langle g_{1},g_{2}\right\rangle$ according to Magam(\cite{MR1484478}), where
  \begin{align*}
 	g_{1}	
 	=&(2, 3, 5, 8, 13, 20, 29)(4, 7, 12, 19, 11, 18, 23)(6, 10, 16, 25, 28, 14, 22)(9, 15, 24, 33, 26, 17, 27)\\
 	&(21, 31, 32,34, 36, 35, 30),
 \end{align*}
 and
 \begin{align*}
 	g_{2}	
 	=&(1, 2, 4)(3, 6, 11)(5, 9, 7)(8, 14, 23)(10, 17, 28)(12, 13, 21)(15, 24, 34)(16, 26, 31)(18, 29, 33)\\
 	&(19, 20, 30)(22, 32, 36)(25, 35, 27).
 \end{align*}

 By the block-transitivity of $G$, we know that $\mathcal{B} = B^{\mathcal{P}(G)}$ for a $6$-subset $B$ of $\mathcal{P}$. Here $B$ is said to be a base block of $\mathcal{D}$. Using the GAP Package
``design"(\cite{GAP4}), the commands ``BlockDesign(v,[B],G)" and ``AllTDesignLambdas(D)", we obtained 46 designs up to isomorphism, among these designs, $\lambda \in \{2,6,12\}$. These designs are summarized in Table \ref{table5}.
 \end{proof}

 	\begin{longtable}{cccccccccc}
 		
 		\caption{Block-transitive $2$-$(6^2,6,\lambda)$ designs with $PSL(2,8)$  \label{table5} }\\
 		
 		\hline
 		Case&Base block $B$&$\lambda$&Case&Base block $B$&$\lambda$\\
 		\hline
 		\endfirsthead
 		\hline
 		Case&Base block $B$&$\lambda$&Case&Base block $B$&$\lambda$\\
 		\hline
 		\endhead
 		\hline
 		\endfoot
	$1$&$\{ 1, 2, 3, 4, 15, 16 \}$&$12$&$24$&$\{ 1, 2, 3, 9, 17, 30 \}$&$12$\\
$2$&$\{ 1, 2, 3, 4, 15, 27 \}$&$12$&$25$&$\{ 1, 2, 3, 9, 18, 22 \}$&$12$\\
$3$&$\{ 1, 2, 3, 4, 16, 24 \}$&$12$&$26$&$\{ 1, 2, 3, 9, 18, 25 \}$&$12$\\
$4$&$\{ 1, 2, 3, 4, 22, 34 \}$&$12$&$27$&$\{ 1, 2, 3, 9, 22, 27 \}$&$12$\\
$5$&$\{ 1, 2, 3, 4, 24, 27 \}$&$12$&$28$&$\{ 1, 2, 3, 9, 22, 34 \}$&$12$\\
$6$&$\{ 1, 2, 3, 4, 25, 34 \}$&$12$&$29$&$\{ 1, 2, 3, 9, 25, 34 \}$&$12$\\
$7$&$\{ 1, 2, 3, 4, 26, 35 \}$&$12$&$30$&$\{ 1, 2, 3, 9, 28, 35 \}$&$12$\\
$8$&$\{ 1, 2, 3, 4, 31, 34 \}$&$12$&$31$&$\{ 1, 2, 3, 10, 12, 24 \}$&$12$\\
$9$&$\{ 1, 2, 3, 4, 32, 34 \}$&$12$&$32$&$\{ 1, 2, 3, 10, 12, 25 \}$&$12$\\
$10$&$\{ 1, 2, 3, 4, 34, 36 \}$&$12$&$33$&$\{ 1, 2, 3, 12, 14, 22 \}$&$12$\\
$11$&$\{ 1, 2, 3, 6, 12, 25 \}$&$12$&$34$&$\{ 1, 2, 3, 12, 17, 35 \}$&$12$\\
$12$&$\{ 1, 2, 3, 6, 15, 23 \}$&$12$&$35$&$\{ 1, 2, 3, 12, 32, 34 \}$&$12$\\
$13$&$\{ 1, 2, 3, 6, 15, 32 \}$&$12$&$36$&$\{ 1, 2, 3, 16, 24, 30 \}$&$12$\\
$14$&$\{ 1, 2, 3, 6, 17, 36 \}$&$12$&$37$&$\{ 1, 2, 3, 16, 24, 32 \}$&$12$\\
$15$&$\{ 1, 2, 3, 6, 24, 26 \}$&$12$&$38$&$\{ 1, 2, 3, 16, 28, 36 \}$&$6$\\
$16$&$\{ 1, 2, 3, 6, 26, 36 \}$&$12$&$39$&$\{ 1, 2, 3, 17, 24, 35 \}$&$12$\\
$17$&$\{ 1, 2, 3, 9, 10, 27 \}$&$12$&$40$&$\{ 1, 2, 3, 17, 25, 31 \}$&$12$\\
$18$&$\{ 1, 2, 3, 9, 10, 32 \}$&$12$&$41$&$\{ 1, 2, 3, 17, 30, 31 \}$&$12$\\
$19$&$\{ 1, 2, 3, 9, 11, 17 \}$&$12$&$42$&$\{ 1, 2, 3, 18, 30, 31 \}$&$12$\\
$20$&$\{ 1, 2, 3, 9, 12, 17 \}$&$12$&$43$&$\{ 1, 2, 3, 25, 27, 32 \}$&$6$\\
$21$&$\{ 1, 2, 3, 9, 12, 34 \}$&$12$&$44$&$\{ 1, 2, 4, 16, 26, 31 \}$&$2$\\
$22$&$\{ 1, 2, 3, 9, 16, 32 \}$&$12$&$45$&$\{ 1, 2, 6, 7, 15, 17 \}$&$12$\\
$23$&$\{ 1, 2, 3, 9, 17, 18 \}$&$12$&$46$&$\{ 1, 2, 6, 16, 18, 36 \}$&$6$\\
 		
 		\hline
 	\end{longtable}

 \begin{remark}\rm{ Let $\mathcal{D}_i $ be the design of Case $i$ of Table \ref{table5}.
  The design $\mathcal{D}_{44} $ is a unique flag-transitive $2$-$(36,6,\lambda)$ design admitting  $PSL(2,8)$ as an automorphism group.}
 \end{remark}
  \begin{lemma}\label{le8}
 	\textup{Let $\mathcal{D} = (\mathcal{P},\mathcal{B})$ be a non-trivial $2$-$(36,6,\lambda)$ design admitting $P\Gamma L(2,8)$ as a block-transitive automorphism group.  Then $\lambda \in \{2,6,9,12,18,36\}$.}
 \end{lemma}
 \begin{proof}
 Let $\al\in \mathcal{P}, G = P\Gamma L(2,8)$, then $G_{\alpha} \cong D_{14}:2$. Let $\mathcal{P} = \{1,2,\cdots,36\}$ and $\mathcal{P}(G)$ be the image of the permutation representations of $G$ on $\mathcal{P}$. Then $\mathcal{P}(G) = \left\langle h_{1},h_{2},h_{3}\right\rangle$ according to Magam(\cite{MR1484478}), where
 \begin{align*}
 	h_{1}	
 	=&(1, 2, 4, 9, 16, 24, 10)(3, 7, 6, 13, 11, 19, 18)(5, 12, 20, 8, 15, 14, 23)(17, 25, 22, 28, 21, 29, 27)\\
 	&(26, 33, 30, 34, 36, 35, 31),
 \end{align*}
 \begin{align*}
	h_{2}	
	=&(1, 3, 8)(2, 5, 6)(4, 10, 7)(9, 17, 26)(11, 19, 14)(12, 21, 30)(13, 22, 31)(15, 23, 16)(18, 27, 34)\\
	&(20, 28, 33)(24, 29, 35)(25, 32, 36),
\end{align*}
and
  \begin{align*}
 	h_{3}	
 	=&(1, 3, 8)(2, 6, 5)(4, 11, 15)(7, 14, 16)(9, 18, 12)(10, 19, 23)(13, 20, 24)(17, 27, 21)(22, 28, 29)\\
 	&(26, 34, 30)(31, 33, 35).
 \end{align*}

 Similarly as Lemma \ref{le7}, we obtained 330 designs up to isomorphism which are listed in Table \ref{table4}(in Appendix). Among these designs, $\lambda \in \{2,6,9,12,18,36\}$.
 \end{proof}

\begin{remark}\rm{Let    $\overline{\mathcal{D}_{i}} $ be the design of Case $i$ of Table \ref{table4}. \begin{itemize}
\item [(1)] In Table \ref{table4}, there are $4$ flag-transitive designs admitting  $G\cong PGL(2,8)$ as an automorphism group up to isomorphism. They are $\overline{\mathcal{D}_{111}}$, $\overline{\mathcal{D}_{116}}$, $\overline{\mathcal{D}_{129}}$ and $\overline{\mathcal{D}_{294}}$. These flag-transitive designs are consistent with the results of \cite[Example 3.9]{montinaro2022flag}.
    \item [(2)]
    Among the designs of   Table \ref{table5} and Table \ref{table4}, we have $\mathcal{D}_{12}\cong \overline{\mathcal{D}_{330}}$, $\mathcal{D}_{33}\cong \overline{\mathcal{D}_{311}}$, $\mathcal{D}_{38}\cong \overline{\mathcal{D}_{111}}$, $\mathcal{D}_{39}\cong \overline{\mathcal{D}_{320}}$, $\mathcal{D}_{42}\cong \overline{\mathcal{D}_{302}}$, $\mathcal{D}_{43}\cong \overline{\mathcal{D}_{129}}$, $\mathcal{D}_{44}\cong \overline{\mathcal{D}_{116}}$ and $\mathcal{D}_{46}\cong \overline{\mathcal{D}_{294}}$ .
               \end{itemize}

}
\end{remark}

 \noindent{\bf Proof of Theorem \ref{mt2}} It follows  immediately from Lemma \ref{le7} and Lemma \ref{le8}.$\hfill\square$
 \medskip

\noindent{\bf Data availability} The datasets generated during and/or analyzed during the current study are available from
the corresponding author on reasonable request.

\section*{Declarations}

\noindent{\bf Conflict of interest} The authors declare that they have no known competing financial interests or personal
relationships that could have appeared to influence the work reported in this paper.

\newpage
\begin{appendix}
	\section*{Appendix}

	 	\begin{longtable}{cccccccccc}
		\caption{Block-transitive $2$-$(6^2,6,\lambda)$ designs with $P\Gamma L(2,8)$  \label{table4} }\\
		
		\hline
		Case&Base block $B$&$\lambda$&Case&Base block $B$&$\lambda$\\
		\hline
		\endfirsthead
		\hline
		Case&Base block $B$&$\lambda$&Case&Base block $B$&$\lambda$\\
		\hline
		\endhead
		\hline
		\endfoot
	1 & \{1, 2, 3, 4, 5, 18\} & 36 & 2 & \{1, 2, 3, 4, 5, 20\} & 36 \\
	3 & \{1, 2, 3, 4, 5, 21\} & 36 & 4 & \{1, 2, 3, 4, 5, 22\} & 36 \\
	5 & \{1, 2, 3, 4, 5, 24\} & 36 & 6 & \{1, 2, 3, 4, 5, 26\} & 36 \\
	7 & \{1, 2, 3, 4, 5, 28\} & 36 & 8 & \{1, 2, 3, 4, 5, 31\} & 36 \\
	9 & \{1, 2, 3, 4, 5, 35\} & 36 & 10 & \{1, 2, 3, 4, 6, 9\} & 36 \\
	11 & \{1, 2, 3, 4, 6, 20\} & 36 & 12 & \{1, 2, 3, 4, 6, 22\} & 36 \\
	13 & \{1, 2, 3, 4, 6, 24\} & 36 & 14 & \{1, 2, 3, 4, 6, 27\} & 36 \\
	15 & \{1, 2, 3, 4, 6, 28\} & 36 & 16 & \{1, 2, 3, 4, 6, 30\} & 36 \\
	17 & \{1, 2, 3, 4, 6, 31\} & 36 & 18 & \{1, 2, 3, 4, 6, 35\} & 36 \\
	19 & \{1, 2, 3, 4, 7, 12\} & 36 & 20 & \{1, 2, 3, 4, 7, 17\} & 36 \\
	21 & \{1, 2, 3, 4, 7, 18\} & 36 & 22 & \{1, 2, 3, 4, 7, 20\} & 36 \\
	23 & \{1, 2, 3, 4, 7, 21\} & 36 & 24 & \{1, 2, 3, 4, 7, 22\} & 36 \\
	25 & \{1, 2, 3, 4, 7, 26\} & 36 & 26 & \{1, 2, 3, 4, 7, 34\} & 36 \\
	27 & \{1, 2, 3, 4, 7, 35\} & 36 & 28 & \{1, 2, 3, 4, 8, 9\} & 36 \\
	29 & \{1, 2, 3, 4, 8, 12\} & 36 & 30 & \{1, 2, 3, 4, 8, 17\} & 36 \\
	31 & \{1, 2, 3, 4, 8, 18\} & 36 & 32 & \{1, 2, 3, 4, 8, 21\} & 36 \\
	33 & \{1, 2, 3, 4, 8, 26\} & 36 & 34 & \{1, 2, 3, 4, 8, 27\} & 36 \\
	35 & \{1, 2, 3, 4, 8, 30\} & 36 & 36 & \{1, 2, 3, 4, 8, 34\} & 36 \\
	37 & \{1, 2, 3, 4, 9, 10\} & 36 & 38 & \{1, 2, 3, 4, 9, 12\} & 36 \\
	39 & \{1, 2, 3, 4, 9, 15\} & 36 & 40 & \{1, 2, 3, 4, 9, 16\} & 36 \\
	41 & \{1, 2, 3, 4, 9, 18\} & 36 & 42 & \{1, 2, 3, 4, 9, 19\} & 36 \\
	43 & \{1, 2, 3, 4, 9, 20\} & 36 & 44 & \{1, 2, 3, 4, 9, 24\} & 36 \\
	45 & \{1, 2, 3, 4, 9, 27\} & 36 & 46 & \{1, 2, 3, 4, 10, 12\} & 36 \\
	47 & \{1, 2, 3, 4, 10, 17\} & 36 & 48 & \{1, 2, 3, 4, 10, 24\} & 36 \\
	49 & \{1, 2, 3, 4, 10, 27\} & 36 & 50 & \{1, 2, 3, 4, 10, 28\} & 36 \\
	51 & \{1, 2, 3, 4, 10, 30\} & 36 & 52 & \{1, 2, 3, 4, 10, 31\} & 36 \\
	53 & \{1, 2, 3, 4, 10, 34\} & 36 & 54 & \{1, 2, 3, 4, 11, 12\} & 36 \\
	55 & \{1, 2, 3, 4, 11, 17\} & 36 & 56 & \{1, 2, 3, 4, 11, 18\} & 36 \\
	57 & \{1, 2, 3, 4, 11, 21\} & 36 & 58 & \{1, 2, 3, 4, 11, 24\} & 36 \\
	59 & \{1, 2, 3, 4, 11, 26\} & 36 & 60 & \{1, 2, 3, 4, 11, 28\} & 36 \\
	61 & \{1, 2, 3, 4, 11, 31\} & 36 & 62 & \{1, 2, 3, 4, 11, 34\} & 36 \\
	63 & \{1, 2, 3, 4, 12, 15\} & 36 & 64 & \{1, 2, 3, 4, 12, 17\} & 36 \\
	65 & \{1, 2, 3, 4, 12, 18\} & 36 & 66 & \{1, 2, 3, 4, 12, 20\} & 36 \\
	67 & \{1, 2, 3, 4, 12, 24\} & 36 & 68 & \{1, 2, 3, 4, 12, 34\} & 36 \\
	69 & \{1, 2, 3, 4, 13, 32\} & 9 & 70 & \{1, 2, 3, 4, 15, 17\} & 36 \\
	71 & \{1, 2, 3, 4, 15, 20\} & 36 & 72 & \{1, 2, 3, 4, 15, 22\} & 36 \\
	73 & \{1, 2, 3, 4, 15, 27\} & 36 & 74 & \{1, 2, 3, 4, 15, 30\} & 36 \\
	74 & \{1, 2, 3, 4, 15, 34\} & 36 & 76 & \{1, 2, 3, 4, 15, 35\} & 36 \\
	77 & \{1, 2, 3, 4, 16, 18\} & 36 & 78 & \{1, 2, 3, 4, 16, 21\} & 36 \\
	79 & \{1, 2, 3, 4, 16, 24\} & 36 & 80 & \{1, 2, 3, 4, 16, 26\} & 36 \\
	81 & \{1, 2, 3, 4, 16, 28\} & 36 & 82 & \{1, 2, 3, 4, 16, 30\} & 36 \\
	83 & \{1, 2, 3, 4, 16, 31\} & 36 & 84 & \{1, 2, 3, 4, 17, 21\} & 36 \\
	85 & \{1, 2, 3, 4, 17, 22\} & 36 & 86 & \{1, 2, 3, 4, 17, 28\} & 36 \\
	87 & \{1, 2, 3, 4, 17, 34\} & 36 & 88 & \{1, 2, 3, 4, 18, 19\} & 36 \\
	89 & \{1, 2, 3, 4, 18, 20\} & 36 & 90 & \{1, 2, 3, 4, 18, 21\} & 18 \\
	91 & \{1, 2, 3, 4, 18, 24\} & 36 & 92 & \{1, 2, 3, 4, 18, 26\} & 36 \\
	93 & \{1, 2, 3, 4, 19, 20\} & 36 & 94 & \{1, 2, 3, 4, 19, 21\} & 36 \\
	95 & \{1, 2, 3, 4, 19, 22\} & 36 & 96 & \{1, 2, 3, 4, 19, 26\} & 36 \\
	97 & \{1, 2, 3, 4, 19, 35\} & 36 & 98 & \{1, 2, 3, 4, 20, 22\} & 18 \\
	99 & \{1, 2, 3, 4, 20, 24\} & 36 & 100 & \{1, 2, 3, 4, 20, 35\} & 36 \\
101 & \{1, 2, 3, 4, 21, 22\} & 36 & 102 & \{1, 2, 3, 4, 21, 27\} & 36 \\
103 & \{1, 2, 3, 4, 21, 28\} & 36 & 104 & \{1, 2, 3, 4, 22, 28\} & 36 \\
105 & \{1, 2, 3, 4, 22, 35\} & 36 & 106 & \{1, 2, 3, 4, 24, 31\} & 36 \\
107 & \{1, 2, 3, 4, 25, 33\} & 36 & 108 & \{1, 2, 3, 4, 26, 31\} & 36 \\
109 & \{1, 2, 3, 4, 26, 34\} & 36 & 110 & \{1, 2, 3, 4, 26, 35\} & 36 \\
111 & \{1, 2, 3, 4, 27, 30\} & 6 & 112 & \{1, 2, 3, 4, 28, 31\} & 18 \\
113 & \{1, 2, 3, 4, 31, 34\} & 36 & 114 & \{1, 2, 3, 4, 31, 35\} & 36 \\
115 & \{1, 2, 3, 4, 34, 35\} & 36 & 116 & \{1, 2, 3, 5, 6, 8\} & 2 \\
117 & \{1, 2, 3, 5, 6, 9\} & 36 & 118 & \{1, 2, 3, 5, 6, 12\} & 36 \\
119 & \{1, 2, 3, 5, 6, 13\} & 36 & 120 & \{1, 2, 3, 5, 6, 17\} & 36 \\
121 & \{1, 2, 3, 5, 9, 10\} & 36 & 122 & \{1, 2, 3, 5, 9, 12\} & 36 \\
123 & \{1, 2, 3, 5, 9, 15\} & 36 & 124 & \{1, 2, 3, 5, 9, 16\} & 36 \\
125 & \{1, 2, 3, 5, 9, 18\} & 36 & 126 & \{1, 2, 3, 5, 9, 19\} & 36 \\
127 & \{1, 2, 3, 5, 9, 20\} & 36 & 128 & \{1, 2, 3, 5, 9, 24\} & 36 \\
129 & \{1, 2, 3, 5, 9, 29\} & 6 & 130 & \{1, 2, 3, 5, 10, 13\} & 36 \\
131 & \{1, 2, 3, 5, 10, 19\} & 18 & 132 & \{1, 2, 3, 5, 10, 20\} & 36 \\
133 & \{1, 2, 3, 5, 10, 26\} & 36 & 134 & \{1, 2, 3, 5, 10, 30\} & 36 \\
135 & \{1, 2, 3, 5, 12, 13\} & 36 & 136 & \{1, 2, 3, 5, 12, 16\} & 36 \\
137 & \{1, 2, 3, 5, 12, 19\} & 36 & 138 & \{1, 2, 3, 5, 13, 15\} & 36 \\
139 & \{1, 2, 3, 5, 13, 16\} & 36 & 140 & \{1, 2, 3, 5, 13, 18\} & 36 \\
141 & \{1, 2, 3, 5, 13, 19\} & 36 & 142 & \{1, 2, 3, 5, 13, 20\} & 36 \\
143 & \{1, 2, 3, 5, 13, 24\} & 36 & 144 & \{1, 2, 3, 5, 13, 30\} & 18 \\
145 & \{1, 2, 3, 5, 15, 16\} & 18 & 146 & \{1, 2, 3, 5, 15, 18\} & 36 \\
147 & \{1, 2, 3, 5, 15, 21\} & 36 & 148 & \{1, 2, 3, 5, 15, 24\} & 36 \\
149 & \{1, 2, 3, 5, 15, 30\} & 36 & 150 & \{1, 2, 3, 5, 15, 31\} & 36 \\
151 & \{1, 2, 3, 5, 15, 33\} & 36 & 152 & \{1, 2, 3, 5, 17, 19\} & 36 \\
153 & \{1, 2, 3, 5, 19, 30\} & 36 & 154 & \{1, 2, 3, 5, 19, 31\} & 36 \\
155 & \{1, 2, 3, 5, 23, 25\} & 18 & 156 & \{1, 2, 3, 5, 23, 32\} & 36 \\
157 & \{1, 2, 3, 5, 26, 30\} & 36 & 158 & \{1, 2, 3, 5, 26, 33\} & 36 \\
159 & \{1, 2, 3, 5, 30, 31\} & 36 & 160 & \{1, 2, 3, 5, 30, 34\} & 36 \\
161 & \{1, 2, 3, 7, 8, 12\} & 36 & 162 & \{1, 2, 3, 7, 8, 17\} & 36 \\
163 & \{1, 2, 3, 7, 8, 25\} & 36 & 164 & \{1, 2, 3, 7, 8, 32\} & 36 \\
165 & \{1, 2, 3, 7, 8, 34\} & 36 & 166 & \{1, 2, 3, 7, 8, 36\} & 36 \\
167 & \{1, 2, 3, 7, 9, 12\} & 36 & 168 & \{1, 2, 3, 7, 9, 16\} & 36 \\
169 & \{1, 2, 3, 7, 9, 21\} & 36 & 170 & \{1, 2, 3, 7, 9, 25\} & 36 \\
171 & \{1, 2, 3, 7, 9, 26\} & 36 & 172 & \{1, 2, 3, 7, 9, 36\} & 36 \\
173 & \{1, 2, 3, 7, 11, 12\} & 36 & 174 & \{1, 2, 3, 7, 11, 25\} & 36 \\
175 & \{1, 2, 3, 7, 11, 34\} & 36 & 176 & \{1, 2, 3, 7, 11, 36\} & 36 \\
177 & \{1, 2, 3, 7, 12, 19\} & 36 & 178 & \{1, 2, 3, 7, 12, 24\} & 36 \\
179 & \{1, 2, 3, 7, 13, 16\} & 36 & 180 & \{1, 2, 3, 7, 13, 22\} & 36 \\
181 & \{1, 2, 3, 7, 13, 25\} & 36 & 182 & \{1, 2, 3, 7, 13, 26\} & 36 \\
183 & \{1, 2, 3, 7, 13, 35\} & 36 & 184 & \{1, 2, 3, 7, 13, 36\} & 36 \\
185 & \{1, 2, 3, 7, 14, 25\} & 36 & 186 & \{1, 2, 3, 7, 14, 32\} & 36 \\
187 & \{1, 2, 3, 7, 14, 34\} & 36 & 188 & \{1, 2, 3, 7, 14, 36\} & 36 \\
189 & \{1, 2, 3, 7, 15, 17\} & 36 & 190 & \{1, 2, 3, 7, 15, 25\} & 36 \\
191 & \{1, 2, 3, 7, 15, 32\} & 36 & 192 & \{1, 2, 3, 7, 15, 34\} & 36 \\
193 & \{1, 2, 3, 7, 16, 20\} & 36 & 194 & \{1, 2, 3, 7, 16, 24\} & 36 \\
195 & \{1, 2, 3, 7, 16, 29\} & 36 & 196 & \{1, 2, 3, 7, 16, 31\} & 36 \\
197 & \{1, 2, 3, 7, 16, 33\} & 36 & 198 & \{1, 2, 3, 7, 17, 19\} & 36 \\
199 & \{1, 2, 3, 7, 17, 27\} & 36 & 200 & \{1, 2, 3, 7, 17, 28\} & 36 \\
	201 & \{1, 2, 3, 7, 17, 29\} & 36 & 202 & \{1, 2, 3, 7, 18, 26\} & 36 \\
	203 & \{1, 2, 3, 7, 18, 27\} & 36 & 204 & \{1, 2, 3, 7, 18, 28\} & 36 \\
	205 & \{1, 2, 3, 7, 18, 30\} & 36 & 206 & \{1, 2, 3, 7, 18, 31\} & 36 \\
	207 & \{1, 2, 3, 7, 18, 33\} & 36 & 208 & \{1, 2, 3, 7, 19, 20\} & 36 \\
	209 & \{1, 2, 3, 7, 19, 21\} & 36 & 210 & \{1, 2, 3, 7, 19, 22\} & 36 \\
	211 & \{1, 2, 3, 7, 19, 34\} & 36 & 212 & \{1, 2, 3, 7, 20, 27\} & 36 \\
	213 & \{1, 2, 3, 7, 20, 29\} & 36 & 214 & \{1, 2, 3, 7, 20, 31\} & 36 \\
	215 & \{1, 2, 3, 7, 20, 33\} & 36 & 216 & \{1, 2, 3, 7, 21, 22\} & 18 \\
	217 & \{1, 2, 3, 7, 21, 24\} & 36 & 218 & \{1, 2, 3, 7, 21, 30\} & 36 \\
	219 & \{1, 2, 3, 7, 22, 31\} & 36 & 220 & \{1, 2, 3, 7, 22, 33\} & 36 \\
	221 & \{1, 2, 3, 7, 24, 26\} & 36 & 222 & \{1, 2, 3, 7, 24, 36\} & 36 \\
	223 & \{1, 2, 3, 7, 25, 27\} & 36 & 224 & \{1, 2, 3, 7, 25, 28\} & 36 \\
	225 & \{1, 2, 3, 7, 25, 29\} & 36 & 226 & \{1, 2, 3, 7, 26, 27\} & 36 \\
	227 & \{1, 2, 3, 7, 26, 35\} & 36 & 228 & \{1, 2, 3, 7, 27, 32\} & 36 \\
	229 & \{1, 2, 3, 7, 28, 32\} & 36 & 230 & \{1, 2, 3, 7, 28, 35\} & 36 \\
	231 & \{1, 2, 3, 7, 29, 32\} & 36 & 232 & \{1, 2, 3, 7, 29, 35\} & 36 \\
	233 & \{1, 2, 3, 7, 30, 34\} & 36 & 234 & \{1, 2, 3, 7, 30, 36\} & 36 \\
	235 & \{1, 2, 3, 7, 31, 32\} & 36 & 236 & \{1, 2, 3, 7, 31, 36\} & 36 \\
	237 & \{1, 2, 3, 7, 32, 33\} & 36 & 238 & \{1, 2, 3, 7, 33, 34\} & 36 \\
	239 & \{1, 2, 3, 7, 33, 36\} & 36 & 240 & \{1, 2, 3, 8, 9, 18\} & 18 \\
	241 & \{1, 2, 3, 8, 9, 28\} & 36 & 242 & \{1, 2, 3, 8, 9, 31\} & 36 \\
	243 & \{1, 2, 3, 8, 10, 17\} & 36 & 244 & \{1, 2, 3, 8, 10, 25\} & 36 \\
	245 & \{1, 2, 3, 8, 10, 32\} & 36 & 246 & \{1, 2, 3, 8, 10, 36\} & 36 \\
	247 & \{1, 2, 3, 8, 12, 13\} & 36 & 248 & \{1, 2, 3, 8, 12, 24\} & 36 \\
	249 & \{1, 2, 3, 8, 13, 26\} & 36 & 250 & \{1, 2, 3, 8, 18, 33\} & 36 \\
	251 & \{1, 2, 3, 9, 11, 12\} & 36 & 252 & \{1, 2, 3, 9, 11, 21\} & 36 \\
	253 & \{1, 2, 3, 9, 11, 25\} & 36 & 254 & \{1, 2, 3, 9, 11, 28\} & 36 \\
	255 & \{1, 2, 3, 9, 11, 31\} & 36 & 256 & \{1, 2, 3, 9, 11, 36\} & 36 \\
	257 & \{1, 2, 3, 9, 12, 20\} & 36 & 258 & \{1, 2, 3, 9, 13, 16\} & 36 \\
	259 & \{1, 2, 3, 9, 13, 19\} & 36 & 260 & \{1, 2, 3, 9, 13, 26\} & 36 \\
	261 & \{1, 2, 3, 9, 13, 31\} & 36 & 262 & \{1, 2, 3, 9, 14, 16\} & 36 \\
	263 & \{1, 2, 3, 9, 14, 19\} & 36 & 264 & \{1, 2, 3, 9, 14, 28\} & 36 \\
	265 & \{1, 2, 3, 9, 17, 23\} & 36 & 266 & \{1, 2, 3, 9, 20, 36\} & 36 \\
	267 & \{1, 2, 3, 9, 23, 34\} & 36 & 268 & \{1, 2, 3, 10, 11, 12\} & 36 \\
	269 & \{1, 2, 3, 10, 11, 25\} & 36 & 270 & \{1, 2, 3, 10, 11, 34\} & 36 \\
	271 & \{1, 2, 3, 10, 12, 13\} & 36 & 272 & \{1, 2, 3, 10, 12, 15\} & 36 \\
	273 & \{1, 2, 3, 10, 12, 16\} & 36 & 274 & \{1, 2, 3, 10, 12, 20\} & 36 \\
	275 & \{1, 2, 3, 10, 12, 23\} & 36 & 276 & \{1, 2, 3, 10, 13, 19\} & 36 \\
	277 & \{1, 2, 3, 10, 13, 28\} & 36 & 278 & \{1, 2, 3, 10, 13, 30\} & 36 \\
	279 & \{1, 2, 3, 10, 13, 31\} & 36 & 280 & \{1, 2, 3, 10, 15, 34\} & 36 \\
	281 & \{1, 2, 3, 10, 15, 36\} & 36 & 282 & \{1, 2, 3, 10, 16, 27\} & 36 \\
	283 & \{1, 2, 3, 10, 16, 30\} & 36 & 284 & \{1, 2, 3, 10, 16, 34\} & 36 \\
	285 & \{1, 2, 3, 10, 17, 23\} & 36 & 286 & \{1, 2, 3, 10, 19, 22\} & 36 \\
	287 & \{1, 2, 3, 10, 19, 33\} & 36 & 288 & \{1, 2, 3, 10, 22, 24\} & 36 \\
	289 & \{1, 2, 3, 10, 22, 25\} & 36 & 290 & \{1, 2, 3, 10, 23, 25\} & 36 \\
	291 & \{1, 2, 3, 10, 23, 32\} & 36 & 292 & \{1, 2, 3, 10, 23, 34\} & 36 \\
	293 & \{1, 2, 3, 10, 23, 36\} & 36 & 294 & \{1, 2, 3, 10, 27, 28\} & 6 \\
	295 & \{1, 2, 3, 10, 30, 31\} & 18 & 296 & \{1, 2, 3, 10, 33, 34\} & 36 \\
	297 & \{1, 2, 3, 11, 12, 20\} & 36 & 298 & \{1, 2, 3, 11, 14, 25\} & 36 \\
	299 & \{1, 2, 3, 11, 14, 32\} & 36 & 300 & \{1, 2, 3, 11, 14, 34\} & 36 \\
301 & \{1, 2, 3, 11, 14, 36\} & 36 & 302 & \{1, 2, 3, 11, 17, 22\} & 12 \\
303 & \{1, 2, 3, 11, 17, 29\} & 36 & 304 & \{1, 2, 3, 11, 20, 36\} & 36 \\
305 & \{1, 2, 3, 11, 22, 25\} & 36 & 306 & \{1, 2, 3, 11, 29, 32\} & 36 \\
307 & \{1, 2, 3, 11, 34, 35\} & 36 & 308 & \{1, 2, 3, 12, 13, 25\} & 36 \\
309 & \{1, 2, 3, 12, 13, 36\} & 36 & 310 & \{1, 2, 3, 12, 15, 18\} & 36 \\
311 & \{1, 2, 3, 12, 15, 19\} & 12 & 312 & \{1, 2, 3, 13, 16, 18\} & 36 \\
313 & \{1, 2, 3, 13, 16, 22\} & 18 & 314 & \{1, 2, 3, 13, 16, 24\} & 36 \\
315 & \{1, 2, 3, 13, 18, 30\} & 36 & 316 & \{1, 2, 3, 13, 18, 35\} & 36 \\
317 & \{1, 2, 3, 13, 19, 20\} & 36 & 318 & \{1, 2, 3, 13, 19, 31\} & 36 \\
319 & \{1, 2, 3, 13, 20, 26\} & 36 & 320 & \{1, 2, 3, 13, 26, 30\} & 12 \\
321 & \{1, 2, 3, 14, 18, 32\} & 36 & 322 & \{1, 2, 3, 14, 19, 20\} & 36 \\
323 & \{1, 2, 3, 14, 20, 36\} & 36 & 324 & \{1, 2, 3, 15, 17, 29\} & 36 \\
325 & \{1, 2, 3, 15, 19, 30\} & 36 & 326 & \{1, 2, 3, 15, 20, 35\} & 36 \\
327 & \{1, 2, 3, 15, 22, 27\} & 6 & 328 & \{1, 2, 3, 15, 29, 35\} & 36 \\
329 & \{1, 2, 3, 18, 26, 30\} & 36 & 330 & \{1, 2, 7, 9, 25, 26\} & 12 \\

	\end{longtable}

\end{appendix}
\end{document}